\documentclass{article}

\usepackage{amsmath,amssymb,amsfonts,bbm,color,bm,nicefrac,
graphicx,enumerate}
\usepackage{amsthm}
\usepackage{epstopdf}
\usepackage{geometry}
\usepackage{caption}
\usepackage[toc,page,title]{appendix}
\usepackage{authblk}
\usepackage{float}

\newtheorem{thm}{Theorem}[section]
\newtheorem{lem}[thm]{Lemma}
\newtheorem{prop}[thm]{Proposition}

\newtheorem{definition}{Definition}[section]
\newtheorem{remark}[thm]{Remark}

\newcommand{\comment}[1]{}

\newcommand{\tmop}[1]{\ensuremath{\operatorname{#1}}}
\newcommand{\tmtextbf}[1]{{\bfseries{#1}}}
\newenvironment{enumeratealpha}{\begin{enumerate}[a{\textup{)}}] }{\end{enumerate}}

\begin{document}
\nocite{*}
\title{Fixation to Consensus on Tree-related Graphs}

\author{Sinziana M. Eckner$^{1}$,
       Charles M. Newman$^{2,3}$
       \\
       $^1$ Courant Institute of Mathematical Sciences, New York,
            NY 10012, USA. \\
       $^2$ Courant Institute of Mathematical Sciences and NYU--Shanghai,
            New York, NY 10012, USA. \\
       $^3$ Department of Mathematics, University of California,
            Irvine, CA 92697, USA.
       }
\maketitle

\begin{abstract}
\noindent
We study a continuous time Markov process whose state space consists of an
assignment of $+1$ or $-1$ to each vertex of a graph $G$.
The graphs that we treat are related to homogeneous trees of degree $K\geq3$,
such as finite or infinite stacks of such trees. The initial spin
configuration is chosen from a Bernoulli product measure with
density $\theta$ of $+1$ spins. The system evolves according to an agreement
inducing dynamics: each vertex, at rate 1, changes its spin value to agree with
the majority of its neighbors. We study the long time behavior of this system
and prove that, if $\theta$ is close enough to 1, the system reaches fixation to
consensus. The geometric percolation-type arguments introduced here may be of
independent interest.
\end{abstract}

\section{Introduction}\label{sec:Introduction}

\indent
In this work we study the long term behavior of continuous time
Markov processes whose
states assign either $+1$ or $-1$ (usually called a spin value) to each
vertex $x$ in a graph $G$.
The graphs $G$ we consider are related to homogeneous
trees of degree $K$ and include infinite stacks of homogeneous trees. These
graphs will be specified in Section \ref{sec:TheModel}, where we will also discuss 
some earlier papers where such stacks of
trees have been studied.
The geometric and percolation theoretic methods we introduce
to carry out our analysis (see especially
Section~\ref{sec:Main} and Appendix~\ref{app:Geometric})
are potentially of independent interest.

We denote by $\sigma_x (t)$
the value of the spin
at vertex $x \in G $ at time $t \geq 0$. Starting from a random initial configuration
$\sigma (0) =\{\sigma_x (0)\}_{x \in G}$ drawn from the independent Bernoulli
product measure

\begin{equation}\label{mu}
 \mu_{\theta} (\sigma_x (0) = + 1) = \theta = 1 - \mu_{\theta} (\sigma_x (0)
   = - 1),
\end{equation}

\noindent
the system then evolves in continuous time according to an agreement inducing
dynamics: at rate 1, each vertex changes its value if it disagrees with more than half of
its neighbors, and tosses a fair coin in the event of a tie.
Our arguments and results easily extend to many other types of dynamics,
as discussed in Remark~\ref{remark1.3} below; these include processes in discrete
time, as in~\cite{KM}, and different rules for tie-breaking.

Our process corresponds to the zero-temperature limit of Glauber dynamics
for a stochastic Ising model with ferromagnetic nearest neighbor interactions
and no external magnetic field (see, e.g., \cite{NNS} or \cite{KRB}).
This process has been studied extensively in the physical and mathematical
literature -- primarily on graphs such as the hyper-lattice $\mathbbm{Z}^d$ and
the homogeneous tree of degree $K$, $\mathbbm{T}_K$.
A physical motivation
is the behavior of a magnetic system following the extreme case of a deep quench,
i.e., when a system has reached equilibrium at infinite temperature
and is instantaneously reduced to zero temperature. For references on
this and related problems see, e.g., \cite{NNS} or \cite{KRB}.
The main focus in the study of this model is the formation and evolution of
boundaries delimiting same spin cluster domains:
these domains shrink or grow or split or coalesce
as their boundaries evolve. An interesting question is whether the system
has a limiting configuration, or equivalently does every vertex
eventually stop flipping? Whether

\begin{equation}\label{sigmainfinity}
\lim_{t\rightarrow\infty}\sigma_x(t)
\end{equation}

\noindent
exists for almost every initial configuration, realization of the dynamics and
for all $x\in G$ in the underlying graph depends on $\theta$
and on the structure of the underlying
graph~$G$.
Nanda, Newman and Stein \cite{NNS}
investigated this question when $G= \mathbbm{Z}^2$ and $\theta=\frac{1}{2}$
and found that in this case the limit does not exist, i.e., every vertex
flips forever. Their work extended an old result
of Arratia~\cite{A}, who showed the same on $\mathbbm{Z}$
for $\theta\neq 0$ or $1$.
One important consequence of the methods of \cite{NNS} is that
$\sigma_x (\infty)$ does exist for almost every initial configuration,
realization of the dynamics and every $x\in G$ if the graph is such that
every vertex has an odd number of neighbors,
such as for example $\mathbb{T}_K$ for $K$ odd.

Another question of interest is whether sufficient bias in the initial
configuration leads the system to reach consensus in the limit. I.e., does
there exist $\theta_{\ast}\in(0, 1)$, such that for $\theta \geq \theta_{\ast}$,

\begin{equation}\label{fixation}
\forall x \in G, \mathbbm{P}_{\theta} (\exists T = T (\sigma (0), \omega, x) < \infty
\text{ so that } \sigma_x (t) = + 1 \text{ for } t \geq T) = 1.
\end{equation}

\noindent
We will refer to \eqref{fixation} as \textbf{fixation to consensus} (of $+1$).
Kanoria and Montanari~\cite{KM} studied fixation to consensus on homogeneous
trees of degree $K \geq 3$ for a process with synchronous time dynamics.
Their process has the same update rules as ours, except that all vertices
update simultaneously and at integer times $t\in \mathbb{N}$.
For each $K$, Kanoria and Montanari defined the consensus threshold
$\rho_{\ast}(K)$ to be the smallest bias in $\rho=2\theta-1$ such that the
dynamics converges to the all $+1$ configuration, and proved upper and lower
bounds for $\rho_{\ast}$ as a function of $K$. Fixation to consensus
was also investigated on $\mathbbm{Z}^d$ for the asynchronous dynamics model.
It was conjectured by
Liggett \cite{L} that fixation to consensus holds there for all $\theta > \frac{1}{2}$.
Fontes, Schonmann and Sidoravicius~\cite{FSS} proved consensus
for all $d\geq 2$ with $\theta_{\ast}$ strictly less than but very close to 1
and Morris~\cite{M} proved that $\theta_{\ast} (d) \to 1/2$ as $d \to \infty$.

In \cite{H} Howard investigated the asynchronous dynamics in detail
on $\mathbbm{T}_3$ and showed how fixation takes place. On this tree graph,
vertices fixate in spin chains
(defined as doubly infinite paths of vertices of the same spin sign).
Though no spin chains are present at time 0 when $\theta = 1/2$, Howard showed
that for any
$\epsilon >0$, there are (almost surely) infinitely many distinct $+1$ and $-1$
spin chains at time $\epsilon$. He also showed the existence of a phase
transition in $\theta$: there exists a critical $\theta_c \in(0,\frac{1}{2})$
such that if $\theta <\theta_c$, $+1$ spin chains do not form almost surely,
whereas if $\theta>\theta_c$ they almost surely form in finite time.
Our work is motivated by that of Howard, but for more general tree-related~graphs.

\section{Statements of Theorems}\label{sec:TheModel}

Let $S^\infty$ denote a doubly infinite stack of homogeneous trees of degree
$K\geq 3$, i.e., the graph with vertex set $\mathbb{T}_K \times \mathbb{Z}$
and edge set specified below.
The main focus of this paper is proving fixation to consensus on $S^\infty$
for the process started with an independent
identically distributed initial configuration of parameter $\theta$.
Such infinite stacks of trees have been studied before in the context of
Bernoulli percolation \cite{GN} and Ising models \cite{NW}. More general nonamenable
graphs have also been studied --- see, e.g., \cite{l}. One motivation for these
studies is that as the parameter of the model varies, the behavior is
sometimes like that on a simple homogeneous tree and sometimes like that
on a simple amenable graph like ${\mathbb Z}^d$.

We express $S^{\infty}$ as

\begin{equation}\label{Sinfinity}
  S^{\infty} = \bigcup_{i = - \infty}^{\infty} S_i,
\end{equation}

\noindent
where $S_i =\mathbbm{T}_K \times \{i\}=\{(u, i) : u \in \mathbbm{T}_K, i\in \mathbb{Z} \}$,
and think of this as a decomposition of the
infinite stack $S^\infty$ into layers $S_i$. Let the edge set of $S^\infty$,
$\mathbb{E}^\infty$, be such that any two vertices $x, y\in S^\infty$ are connected by an
edge $e_{x y} \in \mathbb{E}^\infty$ if and only if:

\begin{enumerate}[i]
 \item  $x=(u_x,k), y=(v_y,k) \in S_k$ for some $k$, and the corresponding $u_x$ and $v_y$
are adjacent vertices in $\mathbb{T}_K$; or
 \item $x=(u_x, k)$ for some $k$ and $y=(u_x, k+1)$; or
\item $x=(u_x, k)$ for some $k$ and $y=(u_x, k-1)$.
\end{enumerate}

For a more detailed description of the Markov process than the one given in Section \ref{sec:Introduction},
we associate to each vertex $x\in S^{\infty}$ a rate~1 Poisson process whose arrival times
we think of as a sequence of clock rings at $x$. We will denote the arrival
times of these Poisson processes by $\{ \tau_{x, n}\}_{n=1, 2, \ldots}$
and take the Poisson processes associated to different
vertices to be mutually independent. We associate to the $(x,n)$'s independent
Bernoulli$(1/2)$ random variables with values $+1$ or $-1$,
which will represent the fair coin tosses to be used in the event of a tie.
Let $\mathbbm{P}_{\tmop{dyn}}$ be the probability measure for the
realization of the dynamics (clock rings and tie-breaking
coin tosses), and denote by $\mathbbm{P}_\theta=
\mu_{\theta} \times \mathbbm{P}_{\tmop{dyn}}$ the joint probability measure on
the space $\Omega$ of initial configurations $\sigma (0)$ and realizations of
the dynamics; an element of $\Omega$ will be denoted~$\omega$.

The main result of this paper is the following theorem, which shows fixation
to consensus for nontrivial $\theta$; its proof is given in
Section~\ref{sec:Main}.
Unlike Kanoria and Montanari \cite{KM}, here we do not attempt to
obtain good
upper bounds on $\theta_{\ast}$ though we expect $\theta_{\ast}$ to
approach $1/2$ with increasing degree $K$. We restrict ourselves to proving
fixation to $+1$ for $\theta$
close enough to 1 with the standard majority update rule: when its clock rings,
each vertex updates to agree with the majority of its neighbors or tosses
a fair coin in the event of a tie. See Remark~\ref{remark1.3} below for other update
rules to which our arguments and results apply.

\begin{thm}\label{theorem1}
  Given $K\geq 3$, there exists $\theta_{\ast} < 1$ such that for $\theta >
  \theta_{\ast}$ the process on $S^\infty$  fixates to consensus.
\end{thm}

The same fixation to consensus result holds for the following graphs,
as stated in Theorem~\ref{theorem2} below, whose proof is also given
in Section~\ref{sec:Main}:

\begin{itemize}
 \item Homogeneous trees $\mathbb{T}_K$ of degree $K\geq 3$.

 \item Finite width stacks of homogeneous trees of degree $K\geq 3$ with free
or periodic boundary conditions. These are graphs, which we will denote
by $S^l_f$ or $S^l_p$, with vertex set
$\mathbb{T}_K \times \{0, 1, \ldots, l-1\}$ and edge set $\mathbb{E}_f$ or
$\mathbb{E}_p$. $\mathbb{E}_f$ and $\mathbb{E}_p$ are defined similarly
to the edge set $\mathbb{E}^\infty$ of $S^\infty$: two vertices $x, y\in S^l_f$ are
connected by an edge $e_{x y} \in \mathbb{E}_f$ if and only if either condition~i
above holds; or

\begin{enumerate}
 \item $x, y\in S_k$ for $1\leq k \leq l-2$ and either condition ii or iii holds; or
 \item $x=(u_x, 0)$ and $y=(u_x, 1)$; or
 \item $x=(u_x, l-1)$ and $y=(u_x, l-2)$.
\end{enumerate}

Any two vertices $x, y\in S^l_p$ are
connected by an edge $e_{x y} \in \mathbb{E}_p$ if and only if either condition
i holds; or

\begin{enumerate}
\setcounter{enumi}{3}
 \item $x, y\in S_k$ for $1\leq k \leq l-2$ and either condition ii or iii holds; or
 \item $x=(u_x, 0)$ and $y=(u_x, 1)$ or $y=(u_x, l-1)$; or
 \item $x=(u_x, l-1)$ and $y=(u_x, l-2)$ or $y=(u_x, 0)$.
\end{enumerate}

 \item Semi-infinite stacks of homogeneous trees of degree $K\geq 3$ with
free boundary conditions. These are graphs, which we will denoted by
$S^{\text{semi}}$, with vertex set $\mathbb{T}_K \times \{0, 1, \ldots\}$ and
edge set $\mathbb{E}^{\text{semi}}$. Two vertices $x, y\in S^{\text{semi}}$ are
connected by an edge $e_{x y} \in \mathbb{E}^{\text{semi}}$ if and only either
condition i holds; or

\begin{enumerate}
\setcounter{enumi}{7}
 \item $x, y\in S_k$ for $1\leq k$ and  either condition ii or iii holds; or
 \item $x=(u_x, 0)$ and $y=(u_x, 1)$.
\end{enumerate}
\end{itemize}

\begin{thm}\label{theorem2}
  Fix $K\geq 3$ and $l\geq 2$ and let $G$ be one of the following graphs:
  $\mathbb{T}_K$, $S^l_f$, $S^l_p$ or $S^{\text{semi}}$. There exists
  $\theta_{\ast} < 1$ such that for $\theta >
  \theta_{\ast}$ the process on $G$ fixates to consensus.
\end{thm}

\begin{remark}\label{remark1.3}
Our results have natural extensions to other dynamics.
Let $N_0$ be the maximum
number of neighbors of a vertex in the graph $G$ where $G$ is 
any of the graphs of Theorem~\ref{theorem2}; for some $M_0>\frac{N_0}{2}$,
we can change (arbitrarily) the update rules for those vertices whose number of
$+1$ neighbors is strictly less than $M_0$, and the conclusions of
Theorem \ref{theorem1} or~\ref{theorem2} remain valid with the same proof.
For large
$N_0$, $M_0$ can be taken much larger than $\frac{N_0}{2}$. For example, on the
infinite stack $S^\infty$ of $K$-trees, $N_0=K+2$ and for $K\geq 5$, one can take
$M_0=K-1=N_0-3$ (as is readily seen from the proof of Theorem \ref{theorem1}).
A special case of this type of extension of our results is to modify the update
rule in the event of a tie: e.g., instead of flipping a fair coin, flip a biased
coin with any bias $p\in [0, 1]$ or do nothing.
We can also change from
two-valued spins to any fixed number $q$ of spin values, say $1, 2, \ldots, q$.
The initial configuration is given by the measure
$ \nu (x \text{ is assigned color } i \text{ at time } 0)= \epsilon_i$
where $i\in \{1, \ldots, q\}$ and $\sum_i \epsilon_i =1$ and the updating is done
via a majority rule, e.g., by a rule that respects majority agreement of neighbors on, say,
color 1. We can then think of color $1$ as the $+1$ spin from before,
and the other $q-1$ colors together representing the $-1$ spin.
If $\epsilon_1$ is close enough to 1, we again obtain fixation to $+1$ consensus.
All our results also apply to the synchronous dynamics of~\cite{KM}.
\end{remark}

\section{Preliminaries}\label{sec:Preliminaries}
In order to prove Theorem~\ref{theorem1} we will show that if we take $\theta$ close enough to 1,
then already at time 0 there are stable structures of $+1$ vertices, which are fixed
for all time. We will choose these structures
to be subsets (denoted $\mathcal{T}_i$)
of the layers $S_i$ in the decomposition of $S^\infty$ such that they
are stable with respect to the dynamics. We will define a set $\mathcal{T}$ as the
union of $\mathcal{T}_i$ for all $i$, and show that for $\theta$ close enough to 1,
the complement of $\mathcal{T}$ is a union of almost surely finite components.

\subsection{A Set of Fixed Vertices in $S^{\infty}$}

\begin{definition}\label{T+lS}
  For $i$ fixed, let $\mathcal{T}_i^{+, l} (t)$ be the union of all subgraphs
  $H$ of $S_i$ that are isomorphic to $\mathbb{T}_l$ with $\sigma_x (t) = +
  1, \forall x \in H$.
\end{definition}

We point out that $\mathcal{T}_i^{+, K - 1} (t)$ is stable for $K\geq 5$,
since every $x \in \mathcal{T}_i^{+, K - 1} (t)$ has $K - 1$ out of $K + 2$
neighbors of spin $+ 1$
and $K-1 > \frac{K+2}{2}$ for $K \geq 5$. Not only is this set stable
with respect to the dynamics on $S^\infty$ as in Theorem \ref{theorem1},
but it's also stable with respect to the dynamics on $S_i$ and the other graphs
of Theorem \ref{theorem2}. Let $\mathcal{T}$ represent the union
of $\mathcal{T}_i^{+, K - 1} (0)$ across all levels $S_i$,~i.e.,

\begin{equation}
\mathcal{T} = \bigcup_{j = - \infty}^{\infty} \mathcal{T}_i,
\end{equation}

\noindent
where for shorthand notation, $\mathcal{T}_i = \mathcal{T}_i^{+, K - 1} (0)$.

If $K\leq 4$, $\mathcal{T}$ as defined above is not stable with respect to
the dynamics. In these cases
the argument will be changed somewhat as discussed in Section~\ref{sec:Main}.

\subsection{Asymmetric Site Percolation on $\mathbbm{T}_K$}
The goal of this subsection is to state and prove a geometric probability estimate,
Proposition \ref{PropIndept}, which concerns
asymmetric site percolation on $\mathbbm{T}_K$ distributed according to the product
measure $\mu_\theta$ with:

\begin{equation}
\mu_{\theta} (\sigma_x = + 1) = \theta = 1 - \mu_{\theta} (\sigma_x = - 1),
   \forall x \in \mathbbm{T}_K.
\end{equation}

\noindent
This equals the distribution of $\sigma(0, \omega)$ restricted to the
layers $S_i$, and
therefore applies to these graphs as well.
The statement and proof of Proposition~\ref{PropIndept} require a series of definitions.
The first of these defines graphical subsets of $\mathbb{T}_K$, whereas
the second concerns probabilistic events for subgraphs of $\mathbb{T}_K$
that have a specific orientation. Later, in the proof of Theorem \ref{theorem1} which is
given in Section~\ref{sec:Main}, Proposition~\ref{PropIndept}
will be applied to certain subsets of $S_i$.

\begin{definition}
\label{TreeGraphs}
\emph{\bf{Certain rooted subtrees of $\mathbb{T}_K$}}
  \noindent
  Let $x, y$ in $\mathbb{T}_K$ be two adjacent vertices, and denote by $A_y
  [x]$ the connected component of $x$ in $\mathbb{T}_K \backslash \{y\}$
  -- see Figure~\ref{fig:mainpart1}.

  \noindent
  Let $x, y, z$ be three adjacent vertices in $\mathbb{T}_K$, such that $x$
  and $z$ are neighbors of $y$. Denote by $A_{x, z} [y]$ the connected
  component of $y$ in $\mathbb{T}_K \setminus \{x \cup z\}$ -- see
  Figure~\ref{fig:mainpart2}.

\end{definition}

\begin{figure}[H]
\centering
\includegraphics{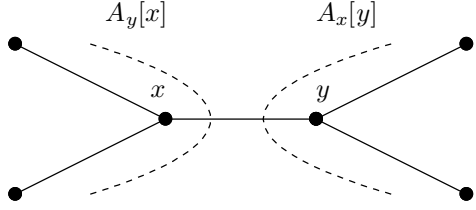}
\caption{$A_y[x]$ and $A_x[y]$ are tree graphs whose roots have coordination number~$K- ~1$}
\label{fig:mainpart1}
\end{figure}

\begin{figure}[H]
\centering
\includegraphics{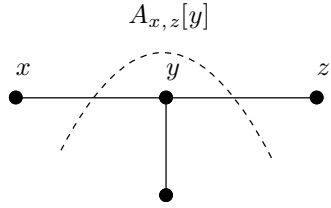}
\caption{$A_{x, z}[y]$ is a tree graph whose root has coordination number $K-2$}\
\label{fig:mainpart2}
\end{figure}

\begin{definition}
\label{K-1-arytrees}
\emph{\bf{Random $(K - 1)$-ary trees of spin $+1$ with a certain \\ orientation}}

  \noindent
  Let $T$ be a deterministic subtree of $\mathbb{T}_K$ with at least two vertices,
  and $v$ be a leaf of $T$; i.e., $v$ has a neighbor $v'$ in $T$ and $(K - 1)$
  neighbors in $\mathbb{T}_K \setminus T$. $\tmop{Tree}^+ [v]$ is the event
  that there is a subgraph $H$ of $A_{v'}[v]$ isomorphic to $\mathbb{T}_{K - 1}$
  and containing $v$, such that $\sigma_u = + 1, \forall u \in H$ --
  see Figure~\ref{fig:mainpart3}.
  \\
  \newline
  \noindent
  Let $T$ be a deterministic subtree of $\mathbb{T}_K$ with at least five vertices,
  and $v$ be a \textbf{2-point} of $T$ (i.e., a vertex of $T$ with exactly two neighbors
  in $T$) that is also  \textbf{good} (i.e., both its neighbors in $T$ are also 2-points
  of $T$). Let $v', w$ be the two neighbors of $v$ in $T$ and let $w'$ be $w$'s
  other neighbor in $T$. $\tmop{Tree}^+ [v, w]$ is the event that there is a
  subgraph $H$ of $A_{v',w} [v] \cup A_{v, w'} [w]$ isomorphic to
  $\mathbb{T}_{K - 1}$ and containing $v$ and $w$, such that
  $\sigma_u = + 1, \forall u \in H$; here
  $A_{v', w} [v] \cup A_{v, w'} [w]$ is the graph with vertex set
  $\mathbb{V}_{A_{v', w} [v]} \cup \mathbb{V}_{A_{v, w'} [w]}$ and edge set
  $\mathbb{E}^\infty_{A_{v', w} [v]} \cup \mathbb{E}^\infty_{A_{v, w'} [w]} \cup e_{v w}$
  ~--~see Figure~\ref{fig:mainpart4}.

\end{definition}

\begin{figure}[H]
\centering
\includegraphics{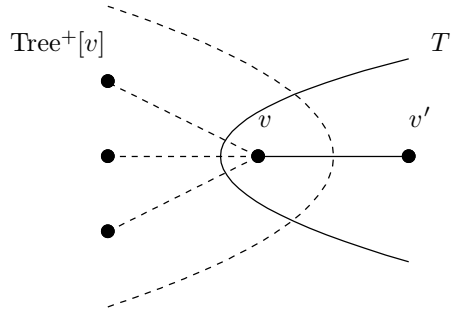}
\caption{The event $\text{Tree}^+[v]$ asserts the existence of a
 random $(K-1)$-ary tree of spin $+1$ that contains a leaf, $v$, of $T$}
\label{fig:mainpart3}
\end{figure}

\begin{figure}[H]
\centering
\includegraphics{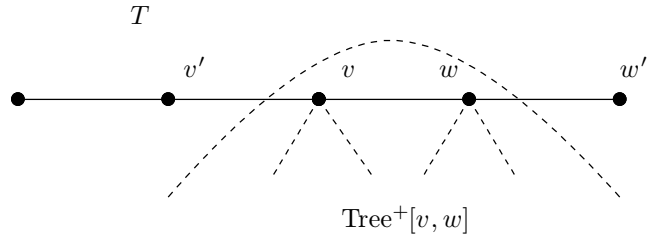}
\caption{The event $\text{Tree}^+[v, w]$ asserts the existence of a random $(K-1)$-ary tree of spin $+1$ that contains a good two 2-point,
 $v$, of $T$, and one of its neighbors, $w$}
\label{fig:mainpart4}
\end{figure}

For distinct leaves $v$ of $T$, the events $\{\tmop{Tree}^+ [v]\}_{v \in T}$
are defined on disjoint subsets of $\mathbbm{T}_K$, and are therefore
independent; they are also identically distributed. The same is true for
$\{\tmop{Tree}^+[v, w]\}_{v, w \in T}$ for disjoint pairs $\{v, w\}$.
The following is essentially the same as Definition \ref{T+lS}, with the only difference
being that here we define the graph $\mathcal{T}^{+, l}$ on $\mathbb{T}_K$, whereas
before we defined the same random graph on $S_i$.

\begin{definition}\label{T+l}
  Let $\mathcal{T}^{+, l}$ be the union of all subgraphs $H$ of
  $\mathbbm{T}_K$ that are isomorphic to $\mathbbm{T}_l$ with $\sigma_x = + 1,
  \forall x \in H$.
\end{definition}

The next proposition estimates the probability that none of the
vertices of a given set $\Lambda$ belong to any random $(K-1)$-ary tree of spin $+1$ (see Definition
\ref{K-1-arytrees}). This proposition is a main ingredient in the proof
of Theorem~\ref{theorem1}.

\begin{prop}\label{PropIndept}
  For any $\lambda \in (0, 1), \exists \theta_{\lambda} \in (0, 1)$ such that
  for $\theta \geq \theta_{\lambda}$ and any deterministic finite nonempty
  subset $\Lambda$ of $\mathbb{T}_K$,
\begin{equation}\label{eqLambda}
 \mu_{\theta} (\Lambda \cap \mathcal{T}^{+, K - 1} = \oslash) \leq
     \lambda^{| \Lambda |} .
\end{equation}

\end{prop}

\begin{proof}

Let $T$ be the minimal spanning tree containing all the vertices of
$\Lambda$. We will call the vertices of $\Lambda$ the \textbf{special}
vertices of $T$. Note that all the leaves of $T$ are special vertices.

We first suppose $|\Lambda| \geq 2$; the simpler case $|\Lambda|=1$ will be handled at the
end of the proof. By the distinctness and disjointness results of Lemma \ref{disjointlemma} 
from the Appendix, there exist constants $\epsilon_1,
\epsilon_2 \in (0, \infty)$ depending only on $K$, such that for each such
tree $T$, one or both of the following is valid:
\begin{enumeratealpha}
  \item there are at least $\epsilon_1 | \Lambda |$ leaves $v$ in $T$, with
  the events $\{\tmop{Tree}^+ (v)\}$ mutually independent, and/or

  \item there are at least $\frac{1}{2} \epsilon_2 | \Lambda |$ edges
  having endpoints $v, w$ in $T$ with $v$ a good special 2-point, and the events
  $\{\tmop{Tree}^+ (v, w)\}_{v, w}$ mutually independent.
\end{enumeratealpha}

Let us first suppose that a) holds. We claim that, for $v$ any leaf of $T$,

\begin{equation}
 \mu_{\theta} (\Lambda \cap \mathcal{T}^{+, K - 1} = \oslash) \leq [1 -
   \mu_{\theta} (\tmop{Tree}^+ [v])]^{\epsilon_1 | \Lambda |} .
\end{equation}
The claim follows from a string of inclusions. First,

\begin{equation}
  \{\Lambda \cap \mathcal{T}^{+, K - 1} = \oslash \}
  \subseteq \bigcap_{v \in T, v \text{ leaf of } T} \{v \notin \mathcal{T}^{+, K
  - 1} \}.
\end{equation}

\noindent
But if $v$ is a leaf of $T$, then
\begin{equation}
  \{v \notin \mathcal{T}^{+, K - 1} \} \subseteq \tmop{Tree}^+ [v]^c,
\end{equation}

\noindent
so that
\begin{equation}
 \bigcap_{v \in T, v \text{ leaf of } T} \{v \notin \mathcal{T}^{+, K
  - 1} \}
  \subseteq \bigcap_{v \in T, v \text{ leaf of } T}
  \tmop{Tree}^+ [v]^c.
\end{equation}

\noindent
Labeling $\epsilon_1 | \Lambda |$ of the leaves in a) as $v_j$,
we restrict the above intersection to the 
leaves $v_j$ of $T$, so that

\begin{equation}
   \bigcap_{v \in T, v \text{ leaf of } T}
  \tmop{Tree}^+ [v]^c
   \subseteq
    \bigcap_{j=1}^{\epsilon_1 | \Lambda |}
  \tmop{Tree}^+ [v_j]^c.
\end{equation}

\noindent
Since the events $\tmop{Tree}^+ [v_j] $ are mutually independent,

\begin{equation}
\mu_{\theta} (\Lambda \cap \mathcal{T}^{+, K - 1} = \oslash) \leq
\prod_{j=1}^{\epsilon_1 |\Lambda|} \mu_{\theta} (\tmop{Tree}^+ [v]^c),
\end{equation}

\noindent
implying the claim.

Alternatively, suppose that b) holds. Now we claim that
\begin{equation}
 \mu_{\theta} (\Lambda \cap \mathcal{T}^{+, K - 1} = \oslash) \leq [1 -
   \mu_{\theta} (\tmop{Tree}^+ [v, w])]^{\frac{1}{2} \epsilon_2 | \Lambda
   |},
\end{equation}

\noindent
where $v$ is a good special 2-point of $T$ and $w$ is one of $v$'s neighbors.
This claim also follows from a string of inclusions. First,

\begin{equation}
  \{\Lambda \cap \mathcal{T}^{+, K - 1} = \oslash \}
  \subseteq \bigcap_{{\{v, w\} \in T, v, w \text{ adj.}} \atop { \text{ and \emph{v} is a good special
  2-point of } T}}
  \{\{v, w\} \notin \mathcal{T}^{+, K
  - 1} \}.
\end{equation}

\noindent
If $\{v, w\}$ are adjacent and $v$ is a good special 2-point of $T$, then
\begin{equation}
  \{\{v, w\} \notin \mathcal{T}^{+, K - 1} \} \subseteq \tmop{Tree}^+ [v, w]^c.
\end{equation}

\noindent
As with the proof of the previous claim, we label $\frac{1}{2}\epsilon_2 |\Lambda|$ of the
pairs of vertices given in b) as $\{v_j, w_j\}$. Then

\begin{equation}
 \{\Lambda \cap \mathcal{T}^{+, K - 1} = \oslash\}
  \subseteq \bigcap_{j=1}^{\frac{1}{2} \epsilon_2 | \Lambda |}
   \tmop{Tree}^+[v_j, w_j]^c.
\end{equation}

\noindent
The second claim follows from the mutual independence of the events
$\tmop{Tree}^+[v_j, w_j]$. The two claims imply (\ref{eqLambda}) for $|\Lambda|\geq 2$
by taking

\begin{equation}
     \lambda > \lambda^*(\theta) = \min \left\{ (1 - \mu_{\theta} (\tmop{Tree}^+
   [v])^{\epsilon_1}, (1 - \mu_{\theta} (\tmop{Tree}^+ [v, w])^{\frac{1}{2}
   \epsilon_2} \right\},
\end{equation}
and using Lemma~\ref{app:lemmaA3} of Appendix~\ref{app:GaltonWatson}.

If $|\Lambda|=1$, suppose the only vertex in $\Lambda$ is 0, a distinguished vertex. Then

\begin{eqnarray}
  \mu_{\theta} ( 0 \notin \mathcal{T}^{+, K - 1})
  & = & 1-\mu_\theta \Big(\bigcup_{j=1}^K \tmop{Tree}^+[a_j] \Big) \\
  & \leq & 1-\mu_\theta (\tmop{Tree}^+[a_1]),
\end{eqnarray}

\noindent
where $a_1, \ldots, a_K$ are the neighbors of 0 and
$\text{Tree}^+[a_j]$ is defined as in Definition \ref{K-1-arytrees}
with $T$ the tree containing only vertices 0 and $a_j$. Then (\ref{eqLambda})
follows in this case by taking $\lambda > \lambda^*(\theta)= 1-\mu_\theta (\tmop{Tree}^+[a_1])$
and using Equation (\ref{taugoesto1}) and Lemma \ref{app:lemmaA1}. This completes the
proof.

\end{proof}

\section{Main Results}\label{sec:Main}
We study the connected components of $S^{\infty} \setminus \mathcal{T}$ as a
subgraph of $S^{\infty}$, and show that if $\theta$ is close enough to 1 these
connected components are finite almost surely. We will show that
each of these finite
connected components of $-1$ vertices shrinks and is eliminated
in finite time leading to fixation of all vertices
to +1.

\begin{definition}
  For any $x \in S^{\infty}$, $D_x$ is the connected component of $x$ in
  $S^{\infty} \backslash \mathcal{T}$: $D_x$ is the set of vertices $y \in
  S^{\infty}$ s.t. $x \overset{S^{\infty} \setminus
  \mathcal{T}}{\leftrightarrow} y$, i.e., there exists a path
  $(x_0 = x, x_1,
  \ldots, x_N = y)$ in $S^\infty$ with every~$x_j \notin \mathcal{T}$.
\end{definition}

\begin{prop}\label{finitecomponents}
  Given $K$, there exists $\theta_{\ast} < 1$ such that for $\theta >
  \theta_{\ast}$, $S^{\infty} \setminus \mathcal{T}$ is a union of almost
  surely finite connected components.
\end{prop}

\begin{proof}
It suffices to show that $D_0$ is finite almost surely, where
0 is a distinguished~vertex in $S^{\infty}$. Since $\mathbb{E}_\theta \left[ \left|
D_0 |] < \infty \right. \right.$ implies $D_0 < \infty$ a.s., it suffices to
show $\mathbb{E}_\theta \left[ \left| D_0 |] < \infty \right. \right.$.

Let $\gamma_N$ represent any site self-avoiding path in $S^\infty$
of length $|\gamma_N| = N \geq 0$ starting
at 0, then by standard arguments

\begin{equation}\label{finitesum}
 \mathbb{E}_\theta \left[ \left| D_0 |] \right. \right. \leq
 \sum_{N = 0}^{\infty}
  \sum_{\gamma_N, | \gamma_N | = N} \mathbb{P}(\gamma_N \in D_0),
\end{equation}

\noindent
where by $\gamma_N \in D_0$ we mean that all the vertices of $\gamma_N$ belong to
$D_0$.

To show the sum is finite we need to bound $\mathbb{P}(\gamma_N \in D_0)$.
Suppose the vertex set of $\gamma_N$ is $\Lambda_1 \cup \ldots \cup \Lambda_J$,
where for each $1 \leq i \leq J$, $\Lambda_i$ is a nonempty subset of
$S_{l_i}$ for some $l_i \in \mathbb{Z}$ with the $l_i$ distinct. We now apply
Proposition \ref{PropIndept} to $\Lambda_i$ in each of the layers $S_{l_i}$, which are isomorphic
to $\mathbb{T}_K$. This shows that for any $\lambda \in (0, 1), \exists \theta_{\lambda}
\in (0, 1)$ such that for $\theta \geq \theta_{\lambda}$,

\begin{equation}
 \mathbb{P}_{\theta} (\Lambda_i \cap \mathcal{T}_{l_i}^{+, K - 1} = \oslash) \leq
   \lambda^{| \Lambda_i |}.
\end{equation}

\noindent
Since the $\Lambda_i$ are subsets of distinct levels $S_{l_i}$ of $S^{\infty}$, the
events $\{\Lambda_i \cap \mathcal{T}_{l_i}^{+, K - 1} = \oslash\}$ are
mutually independent. Therefore for $\theta \geq \theta_{\lambda}$,

\begin{eqnarray}
  \mathbb{P}_\theta (\gamma_N \in D_0)
  & = & \mathbb{P}_\theta \left( \{\Lambda_1 \cap
  \mathcal{T}_{l_1}^{+, K - 1} = \oslash\} \cap \ldots \cap \{\Lambda_J \cap
  \mathcal{T}_{l_J}^{+, K - 1} = \oslash\} \right) \\
  & = & \prod_{i=1}^{J} \mathbb{P}_\theta \left( \{\Lambda_i \cap
  \mathcal{T}_{l_i}^{+, K - 1} = \oslash\} \right) \\
  & \leq & \lambda^{| \Lambda_1 | + \ldots +| \Lambda_J |} \\
  & = & \lambda^N .
\end{eqnarray}

\noindent
Equation~(\ref{finitesum}) and the above bound on
$\mathbb{P}_\theta (\gamma_N \in D_0) $ imply

\begin{eqnarray}
  \mathbb{E}_\theta \left[ \left| D_0 |] \right. \right.
  & \leq & \sum_{N = 0}^{\infty} \lambda^N
   \sum_{\gamma_N, | \gamma_N | = N} 1 \\
  & = & \sum_{N = 0}^{\infty} \rho (N) \lambda^N,
\end{eqnarray}

\noindent
where $\rho (N)$ is the number of self-avoiding paths of length $N$ starting
at 0. It is easy to see that

\begin{equation}
 \rho (N) \leq (K + 2) (K + 1)^{N - 1}.
\end{equation}

\noindent
Thus

\begin{eqnarray}
 \mathbb{E}_\theta \left[ \left| D_0 |] \right. \right. & \leq & (K + 2) \sum_{N =
  0}^{\infty} (K + 1)^{N - 1} \lambda^N .
\end{eqnarray}

\noindent
The proof is finished by choosing $\theta_{\ast}=\theta_{\lambda}$ for
$\lambda<\frac{1}{K+1}$.
\end{proof}

\begin{proof}[Proof of Theorem \ref{theorem1} for $K\geq 5$]

Taking $\theta_{\ast}$ as in Proposition \ref{finitecomponents}, $S^{\infty}
\setminus \mathcal{T}$ is a union of almost surely finite connected components:

\begin{equation}
 S^{\infty} \setminus \mathcal{T} = \bigcup_i D_i,
\end{equation}

\noindent
where the $D_i$'s are nonempty, disjoint and almost surely finite with
$D_i=D_{x_i}$ for some $x_i$.

Fix any $i$; it suffices to show that $D_i$
is eliminated by the dynamics in finite time. By this we mean that there
exists $T_{i} < \infty$ such that for any $y \in D_i, \sigma_y (t) = + 1, \forall t
\geq T_i$, and so the droplet $D_i$ fixates to $+ 1$. We proceed to show this.

For any set $B \subset S^{\infty}$, let
\begin{equation}
 \partial B =\{x \in B \text{ such that
there is an edge } e_{x y}\in \mathbb{E}^\infty \text{ with } y \in B^c \}.
\end{equation}
\noindent
$\partial (D_i^c) \subset
\mathcal{T}$ so $\partial (D_i^c)$ is stable with respect to the dynamics and for
any $x \in \partial (D_i^c)$, $\sigma_x (0) = + 1$.

Since $D_i$ is finite it contains a longest path, $p = (z,\ldots, w)$. Since $p$
cannot be extended to a longer path, $z$ must have $K + 1$ neighbors in $D_i^c$.
When $z$'s clock first rings, $z$ flips to $+ 1$ and fixates for all later times.
This argument can be extended to show $D_i$ is eliminated
(i.e., the $-1$ vertices are all flipped to $+1$) by the dynamics in
finite time as follows. Consider the set of vertices in $D_i$ which have not yet
flipped to $+ 1$ by some time $t$, and take $t$ to infinity. Suppose this
limiting set is nonempty. Since this set is finite, it contains a longest path
$\tilde{p} = (\tilde{z}, \ldots, \tilde{w})$. But now $K + 1$ of $\tilde{z}$'s
neighbors have spin $+ 1$ as $t \rightarrow \infty$, implying that $\tilde{z}$
had no clock rings in
$[T, \infty)$ for some finite $T$. This event has zero probability of occurring,
which contradicts the supposition of a nonempty limit set.
\end{proof}

The proof of Theorem \ref{theorem1} for $K=3$ and 4 is slightly different than
for $K\geq 5$, since for $K=3$ (respectively, $K=4$) the $\mathcal{T}_i$'s of
Definition \ref{T+lS} are not stable
with respect to the dynamics: each vertex $v\in \mathcal{T}_i$ has 2 (resp., 3)
neighbors of spin $+1$, which is always less than a strict majority.
The proof for $K=3$ or 4 requires
a different decomposition of the space $S^\infty$ and definition of stable
subsets. With this purpose in mind, we express $S^\infty$ as

\begin{equation}
 S^{\infty} = \bigcup_{i = - \infty}^{\infty} \tilde{S}_i,
\end{equation}

\noindent
where $\tilde{S}_i =\mathbb{T}_k \times \{2i, 2i+1\}=\{(u, j) : u \in \mathbb{T}_K,
\text{ and } j=2i \text{ or } 2i+1\}$ (see Equation (\ref{Sinfinity}) for a comparison).
We call a vertex $x=(u, 2i)$ or its partner in $\tilde{S}_i$, $\hat{x}=(u, 2i+1)$,
\textbf{doubly open} if both $\sigma_x(0)=+1$ and $\sigma_{\hat{x}}(0)=+1$;
this occurs with probability $\theta^2$. We proceed by defining a set of fixed vertices in
$S^\infty$ in the spirit of Section 2.1.

\begin{definition}
  For $i$ fixed, let $\mathcal{\tilde{T}}_i^{+, l}$ be the union of all subgraphs
  $H$ of $\tilde{S}_i$ that are isomorphic to $\mathbb{T}_l \times \{2i, 2i+1 \}$
  such that $\forall x\in H$, $x$ is doubly open.
\end{definition}

It is easy to see that $\mathcal{\tilde{T}}_i^{+, K - 1}$ is stable for $K=3$
or 4 with respect to the dynamics on $S^\infty$.
Let $\mathcal{\tilde{T}}$ denote the union
of $\mathcal{\tilde{T}}_i^{+, K - 1}$ across all levels $\tilde{S}_i$, i.e.,

\begin{equation}
\mathcal{\tilde{T}} = \bigcup_{i = - \infty}^{\infty} \mathcal{\tilde{T}}_i,
\end{equation}

\noindent
where $\mathcal{\tilde{T}}_i = \mathcal{\tilde{T}}_i^{+, K - 1}$.

\begin{proof}[Proof of Theorem \ref{theorem1} for $K= 3$ and $4$]
We map one independent percolation model, $\sigma_{(u, j)}(0)$ on $S^\infty$
with parameter $\theta$, to another one, $\tilde{\sigma}_{(u, i)}(0)$ on
$S^\infty$ with parameter $\theta^2$, by defining $\tilde{\sigma}_{(u, i)}(0)=+1$
(resp., $-1$) if $(u, 2i)$ is doubly open (resp., is not doubly open).
Propositions \ref{PropIndept} and \ref{finitecomponents} applied to $\tilde{\sigma}$
imply that Proposition \ref{finitecomponents} with $\mathcal{T}$ replaced
by $\mathcal{\tilde{T}}$ is valid for the $\tilde{\sigma}(0)$ percolation model.
The rest of the proof proceeds as in the case for $K \geq 5$.
\end{proof}

\begin{proof}[Proof of Theorem \ref{theorem2}]
The proof proceeds analogously to that of Theorem \ref{theorem1}, except that
the conclusion of Proposition \ref{finitecomponents}, that
$S^\infty \setminus \mathcal{T}$ almost surely has no infinite components
(for $\theta$ close to $1$), is replaced by an analogous result for
$G\setminus \mathcal{T}_G$ with an appropriately defined $\mathcal{T}_G$.
We next specify a choice of $\mathcal{T}_G$ for each of our graphs $G$ and leave
further details (which are straightforward given the proof of Proposition
\ref{finitecomponents}) to the reader.

For $G=\mathbb{T}_K$ with {\it any\/} $K\geq 3$, we simply label
$\mathcal{T}_G=\mathcal{T}^{+, K-1}$ (see Definition \ref{T+l}). For $G=S^\text{semi}$,
$\mathcal{T}_G$ depends on $K$ like it did for $G=S^\infty$ - i.e., for $K\geq5$,
we take
\begin{equation}
 \mathcal{T}_G = \bigcup_{i=0}^\infty \mathcal{T}^{+, K-1}_i,
\end{equation}

\noindent
and for $K=3$ or $4$ we take

\begin{equation}
 \mathcal{T}_G = \bigcup_{i=0}^\infty \tilde{\mathcal{T}}^{+, K-1}_i.
\end{equation}

\noindent
For $G=S_f^l$ or $S_p^l$ with $K\geq 5$, we take

\begin{equation}
 \mathcal{T}_G = \bigcup_{i=0}^{l-1} \mathcal{T}^{+, K-1}_i.
\end{equation}

\noindent
For $G=S_f^l$ or $S_p^l$ with $K=3$ or $4$, the choice of $\mathcal{T}_G$
depends on whether $l$ is even or odd since in the odd case the layers cannot
be evenly paired. If $l$ is even, then we take

\begin{equation}
 \mathcal{T}_G = \bigcup_{i=0}^\frac{l-2}{2} \tilde{\mathcal{T}}^{+, K-1}_i.
\end{equation}

\noindent
For $l$ odd (and $\geq 3$), we pair off the first $l-3$ layers and then use the
final $3$ layers to define
$\tilde{\tilde{\mathcal{T}}}^{+, K-1}$ in which the use
of doubly open sites for $\tilde{\mathcal{T}}^{+, K-1}$ is replaced by
\textbf{triply open} sites; then we take

\begin{equation}
 \mathcal{T}_G = \left( \bigcup_{i=0}^\frac{l-3}{2}
 \tilde{\mathcal{T}}^{+, K-1}_i \right) \cup \tilde{{\tilde{\mathcal{T}}}}^{+, K-1}.
\end{equation}
\end{proof}

\begin{appendices}
\section{Galton-Watson Lemmas}\label{app:GaltonWatson}

The goal of this section is to show that the quantity

\begin{equation} \label{lambda}
\lambda^* (\theta) = \min \left\{ (1 - \mu_{\theta} (\tmop{Tree}^+
[v])^{\epsilon_1}, (1 - \mu_{\theta} (\tmop{Tree}^+ [v, w])^{\frac{1}{2}
\epsilon_2} \right\},
\end{equation}

\noindent
which appears at the end of the proof of Proposition~\ref{PropIndept},
converges to $0$ as $\theta \rightarrow 1$. Here $v$ is a leaf of a subtree
$T$ of $\mathbb{T}_K$, $\{v, w\}$ is a pair of adjacent vertices of $T$ such
that $v$ is a good 2-point (as in Definition~\ref{K-1-arytrees}),
and $\epsilon_1, \epsilon_2$ are fixed constants.

For this purpose we consider independent site percolation on $\mathbb{T}_K$ and
let $a_1, a_2, \ldots, a_K$ denote the $K$ neighbors of 0, a distinguished
vertex in $\mathbb{T}_K$. We associate to each $a_i$ a tree $A_0 [a_i]$
(see Definition \ref{TreeGraphs}), for $i = 1, \ldots, K$ -- see Figure~\ref{fig:app1}.

\begin{figure}[H]
\centering
\includegraphics{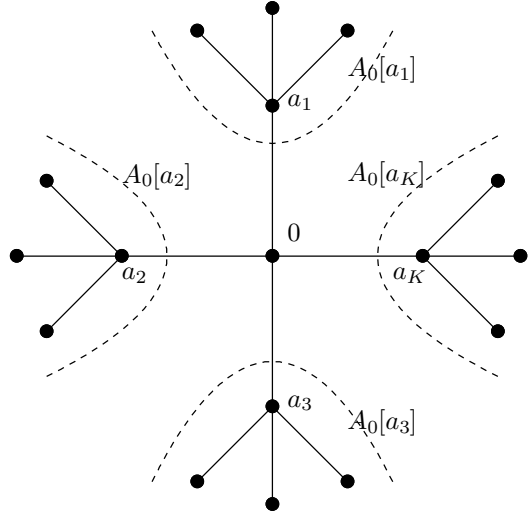}
\caption{$K$-ary tree with labeled vertices and branches}
\label{fig:app1}
\end{figure}

Let $T$ be a subtree of $\mathbb{T}_K$ such that
$T$ contains $a_1$ and $0$ is a leaf of $T$ -- see Figure~\ref{fig:app2}.

\begin{figure}[H]
\centering
\includegraphics{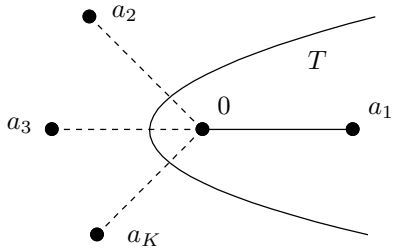}
\caption{$T$ is a subtree of $\mathbb{T}_K$ and 0 is a leaf of $T$}
\label{fig:app2}
\end{figure}

Let $b$ be one of the neighbors of $a_2$ (other than 0) and $T'$ be a subtree
of $\mathbb{T}_K$ containing $b, a_2, 0$ and $a_1$ (but not $a_3, \ldots, a_K$) such that $a_2$ is a
good 2-point of $T'$  -- see Figure~\ref{fig:app3}.

\begin{figure}[H]
\centering
\includegraphics{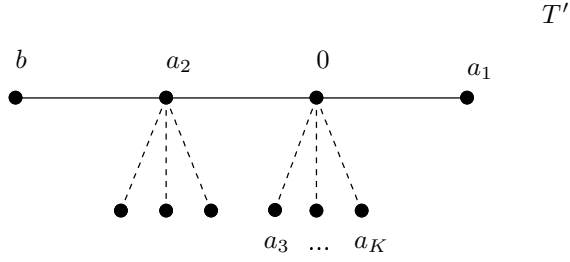}
\caption{$T'$ is a subtree of $\mathbb{T}_K$ that contains $b, a_2, 0$ and $a_1$
 such that $a_2$ is a good 2-point of $T'$}
\label{fig:app3}
\end{figure}

We consider the  events $\tmop{Tree}^+[0]$ with respect to
$T$ and $\tmop{Tree}^+[a_2, 0]$ with respect to $T'$ (see Definition~\ref{K-1-arytrees})
and estimate $\mu_{\theta} (\tmop{Tree}^+[0])$,
$\mu_{\theta} (\tmop{Tree}^+[a_2, 0])$ by analyzing a related Galton-Watson~process.

\begin{definition}
For any vertex $v \in \mathbb{T}_K$, let $C (v)$ denote the $+1$ spin
cluster of $v$, that is, $C (v)$ is the set of vertices $u$ in
$\mathbb{T}_K$ such that the path from $v$ to $u$ (including $v$ and $u$)
includes only vertices $w$, with $\sigma_w = + 1$.
\end{definition}

\noindent
Let
\begin{equation}
Z_n = |v \in A_0 [a_1] \cap C (a_1) :  d(a_1, v) = n|,
\end{equation}

\noindent
where $d (a_1, v)$ represents the graph distance(i.e., the minimum number of edges
) between $a_1$ and $v$.
$Z_0 = 1$ if and only if $\sigma_{a_1} = + 1$, and in general $Z_n$ is the
number of vertices in $A_0 [a_1]$ at distance~$n$ from $a_1$ that are in
$a_1$'s $+1$ spin cluster. $Z_n$ is a Galton-Watson branching process with
offspring distribution $\tmop{Bin} (K - 1, \theta)$.

Let $\mathbb{T}^{\tmop{root}} [x]$ denote a tree with root $x$, such that $x$
has coordination number $K - 2$ and all the other vertices have coordination
number $K-1$.
\noindent
The following definition is close to that of $\tmop{Tree}^+[x]$ (see
Definition \ref{K-1-arytrees} and Figure \ref{fig:mainpart3}), except that here the
$(K-1)$-ary tree in question is rooted.

\begin{definition}
\emph{\bf{Random rooted $(K - 1)$-ary trees of spin $+1$}}

\noindent
Consider two vertices $v, v'\in \mathbb{T}_K$ such that $v'$ is a neighbor of
$v$. Let $\tmop{Part}_{v'}^+ [v]$ denote the event that there exists a subgraph
$H$ of $\mathbb{T}_K$ isomorphic to $\mathbb{T}^{\tmop{root}} [v]$
which contains $v$ and is contained in $A_{v'} [v]$, such that for all $u \in H, \sigma_u = + 1$.
\\
\newline
\noindent
Consider three vertices $x, y, z$ such that $x$ and $z$ are neighbors of
$y$. Let $\tmop{Part}_{x, z}^+ [y]$ be the event that there exists a subgraph
$H$ of $\mathbb{T}_K$ isomorphic to $\mathbb{T}^{\tmop{root}} [x]$ which contains $y$
and is contained in $A_{x, z} [y]$ (see Figure \ref{fig:mainpart2}), such that for all
$u \in H, \sigma_u = + 1$.
\end{definition}

Define $\tau (\theta)$ as
\begin{equation}
\tau (\theta) =\mu_\theta(\tmop{Part}_0^+ [a_1]).
\end{equation}

\noindent
$\tmop{Part}_0^+ [a_i]$ and $\tmop{Part}_0^+ [a_j]$ are independent for $i
\neq j$ by construction. The event $\tmop{Tree}^+[0] $
is equivalent to the spin at $0$ being $+1$ and the vertices
$a_2, \ldots, a_K$ being the roots of $(K-1)$-ary trees of spin $+1$, so that

\begin{equation}\label{taugoesto1}
\mu_{\theta} (\tmop{Tree}^+[0]) = \theta \, \tau(\theta)^{K -1} .
\end{equation}

\begin{lem} \label{app:lemmaA1}
$\tau (\theta) \rightarrow 1$ as $\theta \rightarrow 1$.
\end{lem}

\begin{proof}
The proof is a consequence of Proposition 5.30 from~\cite{LP} (about occurrence of $j$-ary
subtrees in Galton-Watson processes).
\end{proof}

Define $\tilde{\tau} (\theta)$ as

\begin{equation}
\tilde{\tau} (\theta) = \mu_{\theta} (\tmop{Part}_{a_1, a_2}^+ [0]).
\end{equation}

\noindent
The event $\tmop{Tree}^+[a_2, 0]$ is equivalent to
$\{ \tmop{Part}_{a_1, a_2}^+ [0] \cap \tmop{Part}_{b, 0}^+ [a_2] \}$,
so that, by the independence of the events $\tmop{Part}_{a_1, a_2}^+ [0]$ and
$\tmop{Part}_{b, 0}^+ [a_2]$,

\begin{equation}\label{tautildegoesto1}
\mu_\theta(\tmop{Tree}^+[a_2, 0]) = \tilde{\tau}(\theta)^2.
\end{equation}

\begin{lem}\label{app:lemmaA2}
$\tilde{\tau} (\theta) \rightarrow 1$ as $\theta \rightarrow 1$.
\end{lem}

\begin{proof}
This result follows as in the proof of Lemma~\ref{app:lemmaA1}.
\end{proof}

Equations (\ref{taugoesto1}) and (\ref{tautildegoesto1}) imply that
$\mu_{\theta} (\tmop{Tree}^+[0])$ and $\mu_\theta(\tmop{Tree}^+[a_2, 0])$
converge to $1$ as $\theta \rightarrow 1$, which immediately implies:

\begin{lem}\label{app:lemmaA3}
$ \lambda^*(\theta) \rightarrow 0$ as $\theta \rightarrow 1$.
\end{lem}

\section{Geometric Lemmas}\label{app:Geometric}

Let $T$ be a finite tree with $N$ vertices and maximal coordination number
$\leq K$. $N_1 \leq N$ of $T$'s vertices are labeled special, such that all of
$T$'s leaves are special vertices. We remark that in Section~\ref{sec:Main},
we start with $|\Lambda|$ special vertices in $\mathbb{T}_K$ and then $T$ is the
minimal subtree of $\mathbb{T}_K$ that contains all the special vertices.

\begin{lem}\label{Grimmet}
  Let $M_1$ be the number of leaves in $T$, $M_2$ the number of $2$-points
  (vertices with exactly two edges in $T$), $\ldots$, $M_K$ the number of
  $K$-points (vertices with exactly $K$ edges in $T$); $M_1 + \ldots + M_K = N$. Then
  \begin{equation}
  M_i \leq M_1
  \end{equation}
  for $i = 3, \ldots, K$.
\end{lem}
\begin{proof}
The proof can be found, for example, as part of Theorem 8.1 in~\cite{G}.
\end{proof}

\begin{definition}
  Recall that a {\em good\/} 2-point in $T$ is a 2-point both of whose neighbors are 2-points. A
  {\em bad\/} 2-point is a 2-point that is not a good 2-point -- see Figure~\ref{fig:app4}.

\end{definition}

\begin{figure}[H]
\centering
\includegraphics{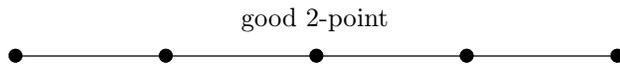}
\caption{A good 2-point}
\label{fig:app4}
\end{figure}

\begin{lem}\label{geometriclemma}
  There exist constants $\epsilon_1, \epsilon_2 \in (0, \infty)$,
  depending only on $K$, such that either:
  \begin{enumeratealpha}
    \item $M_1 \geq \epsilon_1 N_1$, and/or

    \item there are at least $\epsilon_2 N_1$ special good 2-points.
  \end{enumeratealpha}
\end{lem}

\begin{proof}
By Lemma \ref{Grimmet}

\begin{equation}
 \sum_{i = 3}^K M_i \leq (K - 2) M_1,
\end{equation}

\noindent
and since $\sum_1^K M_i = N$,

\begin{eqnarray}
 (K - 1) M_1 + M_2 & = & M_1 + M_2 + (K-2)M_1 \\
 & \geq & \sum_1^K M_i = N.
\end{eqnarray}

\noindent
Thus either $(K - 1) M_1 \geq \frac{N}{K}$ or $M_2 \geq \frac{N (K - 1)}{K}$.
In the first case, since $N \geq N_1,$

\begin{equation}
  M_1 \geq \frac{N}{K (K - 1)} \geq \frac{1}{K (K - 1)} N_1,
\end{equation}

\noindent
and letting $\epsilon_1 = \frac{1}{K (K - 1)}$ gives a).

In the second case $M_2 \geq \frac{N (K - 1)}{K}$, and if a) is not valid with
$\epsilon_1 = \frac{1}{K (K - 1)}$, then $M_1 \leq \frac{N_1}{K(K-1)}$.
To prove b) we need to count the various types of special
vertices in $T$. The set of special vertices is comprised of:
\begin{itemize}

  \item special good 2-points; let {\it Good\/} denote the set of such vertices,
  \item special bad 2-points; let {\it Bad\/} denote the set of such vertices,
  \item special leaves, special 3-points, $\ldots$ , special $K$-points;
  let {\it Other\/} denote
  the set of such vertices.
\end{itemize}

Since $\left| \text{Good} \left| = N_1 - | \text{Bad} | - | \text{Other}
\left| \right. \right. \right.$, we need to upper bound $\left| \text{Other}
\left| \right. \right.$ and $\left| \text{Bad} \left| \right. \right.$.
By Lemma~\ref{Grimmet},

\begin{eqnarray}
  \left| \text{Other} \left| \right. \right.
  & \leq & M_1 + M_3 + \ldots + M_K \\
  & \leq & (K - 2) M_1 \\
  & \leq & \frac{K - 2}{K (K - 1)} N_1.
\end{eqnarray}

\noindent
Now $\left| \text{Bad} \right| \leq \left|
\{\text{all bad 2-points} \} \right|$ and it is easy to see that the latter
is upper bounded by $M_1 + 3M_3 + \ldots + K M_K$. Thus by Lemma \ref{Grimmet},

\begin{eqnarray}
  \left| \text{Bad} \left| \right. \right.
  & \leq & M_1 + 3 M_3 + \ldots + K M_K  \\
  & \leq & M_1 (1+3+\ldots+K) \\
  & \leq & \frac{1}{2} K (K - 1) M_1 \\
  & \leq & \frac{1}{2} N_1,
\end{eqnarray}
since $M_1\leq\frac{N_1}{K(K-1)}$. Thus

\begin{eqnarray}
 \left| \text{Good} \right|
 & = & N_1 - \left|\text{Bad}\right| - \left|\text{Other}\right| \\
 & \geq & N_1 \left(1-\frac{K-2}{K(K-1)}-\frac{1}{2} \right) \\
 & = & N_1 \left( \frac{K^2-3K+4}{2K(K-1)} \right).
\end{eqnarray}

\noindent
We let $\epsilon_2 = \frac{K^2-3K+4}{2K(K-1)}>0$ and so
$\left| \text{Good} \right| \geq \epsilon_2 N_1$.
\end{proof}

\section{Disjointness Lemma}\label{app:Probabilistic}

Consider site percolation on $\mathbb{T}_K$ distributed according to the product
measure $\mu_{\theta}$ with

\begin{equation}
 \mu_{\theta} (\sigma_x = + 1) = \theta = 1 - \mu_{\theta} (\sigma_x = - 1),
   \forall x \in \mathbb{T}_K.
\end{equation}

\noindent
Let $T$ be a finite subtree of $\mathbb{T}_K$ with $2\leq N_1 \leq |T|$ of its vertices labeled
special, such that all the leaves are special. As in Lemma \ref{geometriclemma},
in the following lemma $\epsilon_1$ and $\epsilon_2$ are strictly positive,
finite and depend only on $K$. For the events $\tmop{Tree}^+ [v]$ and
$\tmop{Tree}^+ [v, w]$, see Definition \ref{K-1-arytrees}.

\begin{lem}
\label{disjointlemma}
  \tmtextbf{Disjoint events}

  For each such tree $T$, one or both of the following is valid:
  \begin{enumeratealpha}
    \item there are at least $\epsilon_1 N_1$ leaves $v$ in $T$, with the
    events $\{\tmop{Tree}^+ (v)\}$ mutually independent, and/or

    \item there are at least $\frac{1}{2} \epsilon_2 | \Lambda |$ edges
  having endpoints $v, w$ in $T$ with $v$ a good special 2-point, and the events
  $\{\tmop{Tree}^+ (v, w)\}_{v, w}$ mutually independent.
  \end{enumeratealpha}
\end{lem}

\begin{proof}
Lemma \ref{disjointlemma}.a follows from Lemma \ref{geometriclemma}.a,
since for each of the
$\epsilon_1 N_1$ leaves of $T$ we can define an event $\tmop{Tree}^+ [v]$,
and these events depend on the spins of disjoint sets of vertices
and are therefore mutually independent.

Otherwise, by Lemma \ref{geometriclemma}.b there are at least
$N_3 = \epsilon_2 N_1$ good special 2-points in $T$. These are arranged into
$p\geq1$ nonempty maximal chains of adjacent vertices along $T$. We order the
chains and let $n_i$ denote the number of vertices in the i\textsuperscript{th} chain,
for $i=1, \ldots, p$; $n_1, \ldots, n_p \geq 1$ and $n_1 + \ldots + n_p = N_3$.
We also order the $N_3$ good special 2-points, $\{s_1, s_2, \ldots, s_{N_3}\}$,
so that they are consecutively ordered in each chain.

Suppose $n_i = 1$ for some $i$, and the good special 2-point in this chain is
$s_i^\ast$. Let $w_i$ be one of $s_i^\ast$'s neighbors in $T$ and consider the
event $\tmop{Tree}^+ [s_i^\ast, w_i]$. If $n_i = 2$, the i\textsuperscript{th} chain contains two
adjacent special points $\{s_i^\ast(1), s_i^\ast(2)\}$ and we consider the event
$\tmop{Tree}^+ [s_i^\ast(1), s_i^\ast(2)]$. Generally for the i\textsuperscript{th} chain,
we pair adjacent good special 2-points (other than the last if $n_i$ is odd) so
as to consider $\lfloor \frac{n_i + 1}{2} \rfloor$ events
$\tmop{Tree}^+ [s_i^\ast(j),s_i^\ast(j+1)]$, where the last event is
$\tmop{Tree}^+ [s_i^\ast(n_i),w_i]$ if $n_i$ is odd; these events involve
disjoint sets of vertices and are therefore independent. Thus in total we can construct

\begin{equation}
 \left\lfloor \frac{n_1 + 1}{2} \right\rfloor + \ldots +
 \left \lfloor \frac{n_p + 1}{2} \right \rfloor
 \geq \left\lfloor \frac{N_3}{2} \right\rfloor
\end{equation}

\noindent
mutually independent events.
\end{proof}

\end{appendices}

\textbf{Acknowledgments}: The research reported in this paper was supported in part by NSF
grants 0ISE-0730136 and DMS-1007524.
S.M.E. thanks the Institute of Mathematical Sciences at
NYU--Shanghai for support. The authors thank an anonymous referee for carefully reading the paper and making 
several useful suggestions.

\end{document}